\documentclass[12pt,amstex]{amsart}

\usepackage{mathptmx}
\usepackage{mathrsfs}

\usepackage{verbatim}
\usepackage{url}
\usepackage[all]{xy}
\usepackage{color}
\usepackage{hyperref}

\usepackage{stmaryrd}
\usepackage{epsfig}
\usepackage{amsmath}
\usepackage{amssymb}

\usepackage{amscd}
\usepackage{graphicx}
\usepackage{pstricks}

\theoremstyle{plain}
\newtheorem{thm}{Theorem}[section]
\newtheorem{lem}[thm]{Lemma}

\newtheorem{prop}[thm]{Proposition}
\newtheorem{ques}[thm]{Question}
\newtheorem{cor}[thm]{Corollary}
\theoremstyle{definition}

\theoremstyle{remark}

\numberwithin{equation}{section}

\def \N {\mathbb N}

\def \Z {\mathbb Z}
\def \R {\mathbb R}

\begin{document}
\title{Topological entropy  and IE-tuples of indecomposable continua}

\author[Kato]{Hisao Kato}

\email[Kato]{hkato@math.tsukuba.ac.jp}

\address[Kato]{Institute of Mathematics, University of Tsukuba, Ibaraki 305-8571, Japan}
\keywords{Topological entropy, Cantor set, IE-tuple, chaos, freely tracing property by free chains, indecomposable continuum, inverse limit, $G$-like continuum }

\subjclass[2010]{Primary 37B45, 37B40, 54H20; Secondary 54F15}

\maketitle

\begin{abstract} 
 In [3], by use of ergodic theory method, Blanchard, Glasner, Kolyada and Maass proved that if a map $f:X \to X$ of  a compact metric space $X$ has positive topological entropy, then there is an uncountable $\delta$-scrambled subset of $X$ for some $\delta>0$ and hence the dynamics $(X,f)$ is Li-Yorke chaotic.  In  [18],  Kerr and Li developed local entropy theory and gave a new proof of this theorem. Moreover, by developing some  deep combinatorial tools,  they proved that $X$ contains a Cantor set $Z$ which yields more chaotic  behaviors (see [18, Theorem 3.18]).
In the paper [6], we proved that if $G$ is any graph and a homeomorphism $f$ on a $G$-like continuum $X$ has positive topological entropy, then $X$ has an indecomposable subcontinuum. Moreover, if $G$ is a tree,  there is a pair of two distinct points $x$ and $y$ of $X$ such that the pair $(x,y)$ is an IE-pair of $f$ and 
 the irreducible continuum between $x$ and $y$ in $X$ is an indecomposable  subcontinuum. 
In this paper, we define a new notion of "freely tracing property by free chains" on $G$-like continua and  we prove that a positive topological entropy homeomorphism on a $G$-like continuum admits 
a Cantor set $Z$ such that every tuple of finite points in  $Z$ is an $IE$-tuple of $f$ and $Z$ has the freely tracing property by free chains. Also, by use of this notion,  we 
prove the following theorem: If $G$ is any graph and a homeomorphism $f$ on a $G$-like continuum $X$ has positive topological entropy, then there is a Cantor set $Z$ which is related to both the chaotic  behaviors of  Kerr and Li [18] in dynamical systems  and composants of indecomposable continua in topology. Our main result is Theorem 3.3 whose proof is also a new proof of [6]. Also, we study dynamical properties of continuum-wise expansive homeomorphisms. In this case, we obtain more precise  results concerning continuum-wise stable sets of chaotic continua and IE-tuples.  
\end{abstract}

 \section{Introduction} 
 During the last thirty years or so, many interesting connections between dynamical systems and continuum theory have been studied by many authors \\
(see [1,2,6,7,9-15,17,19,22-25,27,28]). We are interested in the following fact that  chaotic topological dynamics should imply existence of complicated topological structures of underlying spaces.  In many cases,  such continua (=compact connected metric spaces) are indecomposable continua which are central subjects of continuum theory in topology.  We know that many indecomposable continua often appear as chaotic attractors of dynamical systems. 
Also, in many cases, the composants of such indecomposable continua are strongly related to stable or unstable (connected) sets of the dynamics. For instance,   in the theory of dynamical systems and continuum theory, the Knaster continuum (= Smale's horse shoe), the pseudo-arc, solenoids and 
Plykin attractors (=Wada's lakes) etc., are well-known as such indecomposable continua.

In [3], by use of ergodic theory method, Blanchard, Glasner, Kolyada and Maass proved that if a map $f:X \to X$ of  a compact metric space $X$ has positive topological entropy, then there is an uncountable $\delta$-scrambled subset of $X$ for some $\delta >0$ and hence the dynamics $(X,f)$ is Li-Yorke chaotic.  In [18],  Kerr and Li developed local entropy theory and gave a new proof of this theorem. Moreover,  they proved that $X$ contains a Cantor set $Z$ which yields more chaotic  behaviors (see [18, Theorem 3.18]). In [2], Barge and Diamond  showed  that for piecewise monotone
surjections of graphs, the conditions of having positive entropy, containing a horse shoe and the inverse limit space containing an indecomposable subcontinuum are all equivalent.   In [24], Mouron proved that if $X$ is an arc-like continuum which admits a homeomorphism $f$ with positive topological entropy, then $X$ contains an indecomposable subcontinuum. In [6], as an extension of the Mouron's theorem, we proved that if $G$ is any graph and a homeomorphism $f$ on a $G$-like continuum $X$ has positive topological entropy, then $X$ contains an indecomposable subcontinuum. Moreover, if $G$ is a tree,  there is a pair of two distinct points $x$ and $y$ of $X$ such that the pair $(x,y)$ is an IE-pair of $f$ and 
 the irreducible continuum between $x$ and $y$ in $X$ is an indecomposable  subcontinuum. 

In this paper, for any graph $G$ we define a new notion of "freely tracing property by free chains" on $G$-like continua and by use of this notion, we prove that a positive topological entropy homeomorphism on a $G$-like continuum admits 
a Cantor set $Z$ such that every tuple of finite points in  $Z$ is an $IE$-tuple of $f$ and $Z$ has the freely tracing property by free chains. Also, we prove  that the Cantor set $Z$ is related to both  the chaotic  behaviors of  Kerr and Li [18] in dynamical systems  and composants of indecomposable continua in topology. Our main result is Theorem 3.3 whose proof is also a new  proof of [6]. Also, we study dynamical properties of continuum-wise expansive homeomorphisms. In this case, we obtain more precise  results concerning continuum-wise stable sets of chaotic continua  and IE-tuples.

 \section{Definitions and notations} 
\quad\ 
In this paper, we assume that all spaces are separable metric spaces and all maps are continuous. Let $\N$ be the set of natural numbers and $\Z$ the set of integers.  

Let $X$ be a compact metric space and ${\mathcal U}, {\mathcal V}$ be two covers of $X$. Put
\[ {\mathcal U} \vee {\mathcal V} = \{ U \cap V|~U \in {\mathcal U}, V \in {\mathcal V}\}.
\]
The quantity $N({\mathcal U})$  denotes minimal cardinality of subcovers of  ${\mathcal U}$.
Let $f:X \rightarrow X$ be a map and let ${\mathcal U}$ be an open cover of $X$. Put
\[h(f,{\mathcal U}) = \lim_{n \rightarrow \infty} \frac{ \log   N({\mathcal U} \vee
f^{-1}( {\mathcal U}) \vee \ldots \vee f^{-n+1}( {\mathcal U}) )  }{n}.
\]
The \emph{topological entropy of $f$,} denoted by $h(f)$, is the supremum of $h(f, {\mathcal U})$ for all open covers ${\mathcal U}$ of $X$.
The reader may refer to [3,4,5,6,8,18,22-25,27,28] for important  facts concerning 
topological entropy. Positive topological entropy of map  is one of generally accepted definitions of chaos. 

We say that a set $I \subseteq \N$ has \emph{positive density} if
$$\liminf_{n \rightarrow \infty} \frac{|I \cap \{1,2,...,n\}|}{n} >0.$$
Let $X$ be a compact metric space and $f:X \rightarrow X$   a map. Let ${\mathcal A}$
be a collection of subsets of $X$. We say that   ${\mathcal A}$ has an  \emph{ independence set
with positive density} if there exists a set $I \subset \N$  with positive density such that for all finite sets $J \subseteq I$,
and for all $(Y_j) \in \prod _{j\in J}{\mathcal A}$, we have that
$$\bigcap_{j\in J}f^{-j}(Y_j)\neq \emptyset.$$  We observe a simple but important and and useful
fact that if $I$ is an  independence set
with positive density for ${\mathcal A}$ then for all $k \in \Z$, $k+I$ is an  independence set
with positive density for ${\mathcal A}$.
For convenience, we may assume that  $I$ satisfies the condition $(kl)$;  for all $(Y_j) \in \prod _{j\in J}{\mathcal A}$ and any $Y_0\in {\mathcal A}$
$$(kl)~~~~~Y_0\cap \bigcap_{j\in J}f^{-j}(Y_j)\neq \emptyset.$$

We now recall the definition of IE-tuple. Let $(x_1, \ldots, x_n)$ be a sequence of points in $X$.
We say that \emph{ $(x_1, \ldots, x_n)$} is an \emph{IE-tuple for $f$} if whenever $A_1, \ldots ,A_n$ are open sets
containing $x_1, \ldots, x_n$, respectively, we have that the collection ${\mathcal A} =\{A_1,
\ldots, A_n\}$ has an independence set with positive density.  In the case that $n=2$, we use
the term IE-pair. 
We use $IE_k$ to denote the set of all IE-tuples of length $k$.

Let $f: X\to X$ be a map of a compact metric space $X$ with metric $d$ and let $\delta>0$. 
A subset $S$ of $X$ is a $\delta$-{\it scrambled set} of $f$ if $|S|\geq 2$ and for any  $x, y\in S$ with $x\neq y$, then one has 
$$\liminf_{n\to \infty}d(f^n(x),f^n(y))=0 ~~\mbox{and}~~
\limsup_{n\to \infty}d(f^n(x),f^n(y))\geq \delta.$$ 
We say that $f: X\to X$ is {\it Li-Yorke chaotic} if there is an uncountable subset $S$ of $X$ such that  for any  $x, y\in S$ with $x\neq y$, then one has 
$$\liminf_{n\to \infty}d(f^n(x),f^n(y))=0 ~~\mbox{and}~~
\limsup_{n\to \infty}d(f^n(x),f^n(y))>0.$$ 
Also, $f$ has {\it sensitive dependence on initial conditions} 
if there is a positive number  $c >0$ such that 
for any $x\in X$ and any neighborhood $U$ of $x$, one can find $y\in U$ and $n\in \N$ such that $d(f^n(x),f^n(y))\geq  c$. \\

Let $X_i~(i \in \N)$ be a sequence of compact metric spaces  and let
$f_{i,i+1}:X_{i+1} \to X_i$ be a map for each $i\in \N$. The {\it inverse
limit} of the inverse sequence
$\{X_i,f_{i,i+1}\}_{i=1}^{\infty}$ is the space 
$$\varprojlim \{X_i,f_{i,i+1}\}=\{(x_i)_{i=1}^{\infty}~|~ x_i=f_{i,i+1}(x_{i+1}) ~\mbox{for each } i\in \N\} \subset \prod
_{i=1}^{\infty}X_i$$
which has the topology inherited as a subspace of the product space $\prod
_{i=1}^{\infty}X_i$.  

If $f:X \rightarrow X$ is a  map, then we use \emph{ $\varprojlim (X,f)$ to denote
the inverse limit of $X$ with $f$ as the bonding maps}, i.e., 
\[ \varprojlim (X,f) = \left \{ (x_i)_{i=1}^{\infty} \in X^{\N}|~ f(x_{i+1}) = x_i ~(i\in \N) \right \}.
\]
 Let $\sigma_f:\varprojlim (X,f)\to \varprojlim (X,f)$ be the {\it shift homeomorphism} defined by 
$$\sigma_f(x_1,x_2,x_3,....,)=(x_2,x_3,....,).$$

 A \emph{ continuum} is a compact connected metric space. We say that a continuum is \emph{nondegenerate}
if it has more than one point. A continuum is \emph{indecomposable} (see [19,20,23,26])
if it is nondegenerate  and it is not the union of two proper subcontinua.  
For any continuum $H$, the set $c(p)$ of all points of the  continuum $H$, which can be joined with the point $p$ by a proper subcontinuum of $H$, is said to be the {\it composant} of the point $p\in H$ 
(see [20, p.208]). Note that for an indecomposable continuum $H$, the following are equivalent;
\begin{enumerate}
\item 
the two points $p, q$ belong to same composant of $H$; 
\item
$c(p)\cap c(q)\neq  \emptyset$; 
\item
$c(p)=c(q)$.
\end{enumerate}
So, we know that if $H$ is an indecomposable continuum, the family $$\{c(p)|~p\in H\}$$ of all composants of $H$ is a family of uncountable mutually disjoint sets $c(p)$ which are connected and dense $F_{\sigma}$-sets in $H$ (see [20, p.212, Theorem 6]). Note that a (nondegenerate) continuum $X$ is indecomposable if and only if 
there are three distinct points of $X$ such that any subcontinuum of $X$ containing any two points of 
the three points coincides with $X$, i.e., $X$ is irreducible between any two points of the three points. 

Let $H$ be an indecomposable continuum. We say that a subset $Z$ of $H$ is  \emph{vertically embedded} to composants of $H$ if no two of points of $Z$ belong to the same composant of $H$, i.e., if $x, y$ are any   distinct points of $Z$ and $E$ is any subcontinuum of $H$ containing $x$ and $y$, then $E=H$.   

A map $g$ from $X$ onto $G$ is an \emph{$\epsilon$-map}~$(\epsilon >0)$
if for every $y\in G$, the diameter of $g^{-1}(y)$ is less than $\epsilon$. A continuum $X$ is \emph{ $G$-like} if for every $\epsilon>0$ there is an $\epsilon$-map from $X$ onto $G$.  For any finite polyhedron $G$, $X$ is $G$-like if and only if $X$ is homeomorphic to an inverse limit of an inverse sequence of $G$. Arc-like continua are those which are $G$-like for $G=[0,1]$.  
Our focus in this article is on $G$-like continua where $G$ is a graph (= connected 1-dimensional compact polyhedron). A graph $G$ is a {\it tree} if $G$ contains no simple closed curve. 
A continuum $X$ is \emph{tree-like} if for any 
$\epsilon >0$ there exist a tree  $G_{\epsilon}$ and an $\epsilon$-map from $X$ onto $G_{\epsilon}$. In this case,  $G_{\epsilon}$ depends on $\epsilon$.
 If ${\mathcal G}$ is a collection of subsets of $X$, then the \emph{ nerve} $N({\mathcal G})$ of ${\mathcal G}$  is the polyhedron whose
vertices are elements of ${\mathcal G}$ and there is 
a simplex $<g_1, g_2,...,g_k>$ with  distinct vertices $g_1, g_2,...,g_k$ if 
$$\bigcap_{i} g_i \neq \emptyset.$$  In this paper, we consider the only case that  nerves are graphs.  

 If $\{C_1, \ldots, C_n\}$ is a subcollection of ${\mathcal G}$ we call 
it a \emph{chain} if $C_i \cap C_{i+1} \neq \emptyset$ for $1\le i <n$ and
$\overline{C_i }\cap \overline{ C_j }\neq \emptyset $ implies that $|i - j| \le 1$. We say
that $\{C_1, \ldots, C_n\}$ is a \emph{free chain in ${\mathcal G}$} if it is a chain and, moreover,
for all $1 < i <n$ we have that $C \in {\mathcal G}$ with $\overline{C} \cap \overline{C_i} \neq \emptyset$ implies
that $C=C_{i}$, $C=C_{i-1}$ or $C=C_{i+1}$. By the \emph{mesh} of a finite collection ${\mathcal G}$ of sets, we means the largest 
 of diameters of elements of ${\mathcal G}$. 
Note that for a graph $G$, a continuum $X$ is a $G$-like if and only if 
for any $\epsilon >0$, there is a finite open cover ${\mathcal G}$ of $X$ such that $N({\mathcal G})=G$ (which means that $N({\mathcal G})$ and $G$ are homeomorphic) and the mesh of ${\mathcal G}$ is less than $\epsilon$. The Knaster continuum (= Smale's horse shoe) and the pseudo-arc are arc-like continua, solenoids are circle-like continua and Plykin attractors are $(S_1\vee S_2\vee \cdots \vee S_m)$-like continua, where 
$S_1\vee S_2\vee \cdots \vee S_m~(m\geq 3)$ denotes the one point union of $m$ circles $S_i$. Such spaces are typical indecomposable continua. 
The reader may refer to [20] and [26] for standard facts concerning 
continuum theory.

Let $X$ be a continuum and $m\in \N$. Suppose that 
$A_i~(1 \leq i \leq m)$ are $m$ (nonempty) open sets in $X$ and $x_i~(1 \leq i \leq m)$ are $m$ distinct points of $X$. We identify the order $A_1\to A_2\to \cdots \to A_m$ and the converse order $A_m\to A_{m-1}\to \cdots \to A_1.$ Then we consider the equivalence class 
$$[A_1\to A_2\to \cdots \to A_m]=
\{A_1\to A_2\to \cdots \to A_m; A_m\to A_{m-1}\to \cdots \to A_1\}.$$ 
Suppose that ${\mathcal G}$ is a finite open cover 
of $X$. 
We say that a chain $\{C_1, \cdots, C_n\}
 \subseteq {\mathcal G}$ {\it follows from the pattern} 
$[A_1\to A_2\to \cdots \to A_m]$  if there exist  $$1\leq k_1 < k_2 <\cdots <k_m \leq n~~ {\mbox or}~~
1\leq k_m < k_{m-1} <\cdots <k_1 \leq n$$ such that 
$C_{k_i}\subset A_i$ for each $i=1,2,...,m$.
In this case, more precisely  we say that the chain $[C_{k_1}\to C_{k_2}\to \cdots \to C_{k_m}]$  follows from the pattern  $[A_1\to A_2\to \cdots \to A_m].$
 Similarly, we say that a chain $\{C_1, \ldots, C_n\}
 \subseteq {\mathcal G}$ {\it follows from the pattern}  
$[x_1\to x_2\to \cdots \to x_m]$  if there exist 
$$1\leq k_1 < k_2 <\cdots <k_m \leq n ~~{\mbox or}~~ 
1\leq k_m < k_{m-1} <\cdots <k_1 \leq n$$
 such that 
$x_i\in C_{k_i}$ 
for each $i=1,2,...,m$, where   
$$[x_1\to x_2 \to  \cdots \to x_m]=\{x_1\to x_2 \to \cdots \to x_m;~x_m \to x_{m-1} \to \cdots \to x_1\}.$$
 More precisely,   we say that the chain $[C_{k_1}\to C_{k_2}\to \cdots \to C_{k_m}]$  follows from the pattern $ [x_1\to x_2 \to  \cdots \to x_m].$\\

Let $Z$ be a subset of a $G$-like continuum $X$. We say that $Z$ has {\it the freely tracing property by (resp. free) chains} if for any $\epsilon >0$, any $m\in \N$ and any order $x_1\to x_2 \to \cdots \to x_m$ of any $m$ distinct points 
$x_{i}~(i=1,2,...,m)$ of $Z$, there is an open cover ${\mathcal U}$ of $X$ such that the  
mesh of ${\mathcal U}$ is less than $\epsilon$, the nerve  $N({\mathcal U})$ of ${\mathcal U}$ is $G$ and there is a (resp. free) chain in ${\mathcal U}$ which follows from the pattern $[x_1\to x_2 \to \cdots \to x_m]$.\\

Example 1.  (1)~Let $X=[0,1]$ be the unit interval and $D$ a subset of $X$. If $|D|\geq 3$, $D$ does not have the freely tracing property by chains.\\
(2)~ Let $X=S^1$ be the unit circle and 
$D$ a subset of $X$. If $|D|\leq 3$, then $D$ has  the freely tracing property by free chains. If $|D|\geq 4$, then  $D$ does not have the freely tracing property by chains.
\\

For the case that $X$ is a tree-like, we obtain the following proposition.

\begin{prop}
Let $X$ be a tree-like continuum and let $D$ be a subset of $X$ with $|D|\geq 3$. Then the following are equivalent.\begin{enumerate}
\item
For any order $x_1\to x_2 \to x_3$ of three distinct points 
$x_{i}~(i=1,2,3)$ of $D$ and any $\epsilon >0$, there is an open cover ${\mathcal U}$ of $X$ such that the  
mesh of ${\mathcal U}$ is less than $\epsilon$, the nerve  $N({\mathcal U})$ of ${\mathcal U}$ is a tree and there is a chain in ${\mathcal U}$ which follows from the pattern $[x_1\to x_2 \to x_3]$.
\item 
$D$ has the freely tracing property by chains;  for any $\epsilon >0$, any $m\in \N$ and any order $x_1\to x_2 \to \cdots \to x_m$ of any $m$ distinct points 
$x_{i}~(i=1,2,...,m)$ of $D$, there is an open cover ${\mathcal U}$ of $X$ such that the  
mesh of ${\mathcal U}$ is less than $\epsilon$, the nerve  $N({\mathcal U})$ of ${\mathcal U}$ is a tree and there is a chain in ${\mathcal U}$ which follows from the pattern $[x_1\to x_2 \to \cdots \to x_m]$.  
\item 
The minimal continuum $H$ in $X$ containing $D$ is indecomposable and 
 no two of points of $D$ belong to the same composant of $H$, i.e., $D$ is vertically embedded to composants of $H$.
\end{enumerate}
\end{prop}
\begin{proof}
First, we will show that
(1) implies (3).  Consider the family ${\mathcal K}$ of all subcontinua of $X$ containing $D$. Since 
 $X$ is a tree-like continuum, the intersection 
$$H=\bigcap \{K\in {\mathcal K}\}$$ 
is the unique minimal subcontinuum containing $D$. Suppose, on the contrary, that $H$ is decomposable. 
Then there are proper subcontinua $H_1, H_2$ of $H$ with $H=H_1\cup H_2.$ Since $H$ is the minimal continuum containing $D$, we can choose $x, y\in D$ such that $x\in H_1-H_2, y\in H_2-H_1$. Also let $z$ be a point of $D$ with $z\neq x, z\neq y$.  We may assume that 
$z\in H_1$. Choose $\epsilon >0$ with $\epsilon <d(y,H_1)$. By (1), we can choose  an open cover ${\mathcal U}$ of $X$ such that the  
mesh of ${\mathcal U}$ is less than $\epsilon$, the nerve  $N({\mathcal U})$ of ${\mathcal U}$ is a tree and there is a chain in ${\mathcal U}$ which follows from the pattern $[x\to y \to z]$. Since $H_1$ is connected, 
the family $$\{U\in {\mathcal U}|~U\cap H_1\neq \emptyset\}$$
contains a chain from $x$ to $z$. Then 
we have a circular chain in  ${\mathcal U}$. Since $N({\mathcal U})$ is a tree, this is a contradiction. Hence $H$ is indecomposable. 

Suppose, on the contrary, that there are two distinct points $x,y\in D$ which are contained in the same composant of $H$. Choose a proper subcontinuum $J$ of $H$ containing $x,y$. Let $z$ be any point of $D$ with 
$z\neq x$ and $z\neq y$. Let $\epsilon >0$ be a sufficiently small positive number. By (1), there is an open cover ${\mathcal U}$ of $X$ such that the  
mesh of ${\mathcal U}$ is less than $\epsilon$, the nerve  $N({\mathcal U})$ of ${\mathcal U}$ is a tree and there is a chain in ${\mathcal U}$ which follows from the pattern $[x\to z \to y]$. Since $N({\mathcal U})$ is a tree and $x, y \in J$, we see that $d(z,J)<\epsilon$. Since $\epsilon$ is arbitrary small, then we see that $z\in J$. Hence $D\subset J$. Since $H$ is the minimal continuum containing $D$, this is a contradiction.  Consequently, no two of points of $D$ belong to the same composant of $H$.

Next we will show that (3) implies (1). We assume that the minimal continuum $H$ in $X$ containing $D$ is indecomposable and 
 no two of points of $D$ belong to the same composant of $H$. Let $a, b, c$ be any three distinct points of $D$. We consider 
the order $a\to b \to c$. Let $\epsilon >0$ be sufficiently small so that $\epsilon < \min \{d(a,b), d(b,c), d(c,a)\}$. Since $X$ is tree-like, we have a finite  open cover   $ {\mathcal V}$ of $X$ such that the  
mesh of ${\mathcal V}$ is less than $\epsilon$, the nerve  $N({\mathcal V})$ of ${\mathcal V}$ is a tree. 
For any $x \in X$, let $V_x$ be an element of $ {\mathcal V}$ containing the point $x$. 

We consider the following cases.\\
Case(i):~$V_b$ separates the vertices $V_a$ from $V_c$ in the nerve $N({\mathcal V})$. \\
In this case, we can easily see that 
 there is a chain in ${\mathcal V}$ which follows from the pattern $[a\to b \to c]$.\\
Case(ii):~$V_b$ does not separate the vertices $V_a$ from $V_c$ in the nerve $N({\mathcal V})$. \\
Since $N({\mathcal V})$ is  a tree, 
we can choose subfamilies $ {\mathcal V}'$ and $ {\mathcal V}"$ of $ {\mathcal V}$ such that 
$ {\mathcal V}={\mathcal V}'\cup {\mathcal V}"$, $ {\mathcal V}'\cap {\mathcal V}"=\{V_b\}$,  the nerves   
$N({\mathcal V}')$ and $N({\mathcal V}")$ are connected (tree), and $N({\mathcal V}')$ contains $V_a$ and $V_c$. 
Put  
$$Y=X-\bigcup {\mathcal V}",$$
where 
$$\bigcup {\mathcal V}"=\bigcup \{V|V\in {\mathcal V}"\}.$$ 
Consider the component $Y_a$ containing $a$ in $Y$ and the component $Y_c$ containing $c$ in $Y$. 
Then we see that $Y_a \cap Y_c=\emptyset$. Suppose, on the contrary, that $Y_a \cap Y_c\neq \emptyset$ and hence 
$Y_a=Y_c$. Since $X$ is a tree-like continuum, we see that  $H\cap Y_a$ is a proper subcontinuum of $H$ containing $a$ and $c$.  This implies that $a$ and $c$  belong to the same composant of $H$. This is a contradiction. Hence $Y_a \cap Y_c=\emptyset$. Since $Y_a$ is a component of $Y$, we can choose a sufficiently small closed and open neighborhood $Y_1$ of $Y_a$ in $Y$ such that 
$Y_1\cap Y_c=\emptyset$. Put $Y_2=Y-Y_1$. 
Consider the following families $ {\mathcal V}_1$ and $ {\mathcal V}_2 $ of open sets of $Y$ defined by 
$$ {\mathcal V}_1= {\mathcal V}'|Y_1, ~~ {\mathcal V}_2= {\mathcal V}'|Y_2,$$  
where ${\mathcal V}'|Y_1=\{V\cap Y_1|~V\in {\mathcal V}'\}$.
Put $$ {\mathcal U}' ={\mathcal V}_1\cup {\mathcal V}_2 \cup {\mathcal V}".$$
By use of the cover $ {\mathcal U}'$ of $X$, we can easily construct the desired open cover  ${\mathcal U}$ of $X$ such that the  
mesh of ${\mathcal U}$ is less than $\epsilon$, the nerve  $N({\mathcal U})$ of ${\mathcal U}$ is a tree and there is a chain in ${\mathcal U}$ which follows from the pattern $[a\to b \to c]$.  

Note that it is trivial that (2) implies (1).
Finally, we prove that (1) implies (2). By the induction on $m$, we prove the implication.  In the statement of $(2)$,  the case $m=3$ is the case of $(1)$.  We assume that 
the statement of $(2)$ is true for the case $m\geq 3$. Let  $x_1\to x_2 \to x_3\to \cdots \to x_m\to x_{m+1}$ be any order of 
 distinct $m+1$ points $x_{i}~(i=1,2,...,m+1)$ of $D$ and $\epsilon >0$.  By induction, there is a finite  open cover   $ {\mathcal V}$ of $X$ such that the  
mesh of ${\mathcal V}$ is less than $\epsilon$, the nerve  $N({\mathcal V})$ of ${\mathcal V}$ is a tree and there exists a chain 
 in ${\mathcal V}$ which follows from the pattern $[x_1\to x_2 \to x_3\to \cdots \to x_m]$. 
 As above, we consider the following cases.\\
Case(i):~$V_{x_{m}}$ separates the vertices $V_{x_i}~(i=1,2,...,m-1)$ from $V_{x_{m+1}}$ in the nerve $N({\mathcal V})$. \\
 In this case, we can easily see that 
 there is a chain in ${\mathcal V}$ which follows from the pattern $[x_1\to x_2 \to x_3\to \cdots \to x_m\to x_{m+1}]$.\\
Case(ii):~$V_{x_{m}}$ does not separate the vertices $V_{x_i}~(i=1,2,...,m-1)$ from $V_{x_{m+1}}$ in the nerve $N({\mathcal V})$.  \\
Since $N({\mathcal V})$ is  a tree, 
we can choose subfamilies ${\mathcal V}'$ and ${\mathcal V}"$ of ${\mathcal V}$ such that ${\mathcal V}={\mathcal V}'\cup {\mathcal V}"$,
${\mathcal V}'\cap {\mathcal V}"=\{V_{x_{m}}\}$,  the nerves   $N({\mathcal V}')$ and $N({\mathcal V}")$ are connected (tree) and $N({\mathcal V}')$ contains the vertices $V_{x_i}~(i=1,2,...,m-1)$ and $V_{x_{m+1}}$.
As above, we put  
$$Y=X-\bigcup {\mathcal V}".$$
Note that (1) and (3) are equivalent and hence we can use the conditions of (3). By the arguments as above,  we can choose a closed and open set $Y'$ of $Y$ containing 
$x_{m+1}$ such that $Y'$ does not contain any $x_i~(i=1,2,...,m-1)$. Put $Y"=Y-Y'$. 
Consider the following families $ {\mathcal V}_1$ and $ {\mathcal V}_2 $ of open sets of $X$ defined by 
$$ {\mathcal V}_1= {\mathcal V}'|Y', ~~ {\mathcal V}_2= {\mathcal V}'|Y".$$  
Put $$ {\mathcal U}' ={\mathcal V}_1\cup {\mathcal V}_2 \cup {\mathcal V}".$$
By use of the cover $ {\mathcal U}'$ of $X$, we have the desired open cover  ${\mathcal U}$ of $X$ such that the  
mesh of ${\mathcal U}$ is less than $\epsilon$, the nerve  $N({\mathcal U})$ of ${\mathcal U}$ is a tree and there is a chain in ${\mathcal U}$ which follows from the pattern $[x_1\to x_2 \to x_3\to \cdots \to x_m\to x_{m+1}]$. 
This completes the proof. 
\end{proof}

 \section{Topological entropy on $G$-like continua and Cantor sets which have the freely tracing property by free chains} 
\quad\

In [3], by use of ergodic theory method, Blanchard, Glasner, Kolyada and Maass proved that if a map $f:X \to X$ of a  compact metric space $X$ has positive topological entropy, then there is an uncountable $\delta$-scrambled set of $f$ for some $\delta >0$ and hence  the dynamics $(X,f)$ is Li-Yorke chaotic. 
In [8], Huang and Ye studied local entropy theory and they gave a characterization of positive topological entropy by use of entropy tuples. Moreover, in [18], by use of local entropy theory (IE-tuples), Kerr and Li proved the following more precise theorem.

\begin{thm}
{\rm ([18, Theorem 3.18])}
Suppose that $f:X \to X$ is a positive topological  entropy map on a compact metric space $X$, and $x_1,x_2,...,x_m~(m\geq 2)$ are finite distinct points of $X$ such that the tuple $(x_1,x_2,...,x_m)$ is an $IE$-tuple of $f$. If $A_i~(i=1,2,...,m)$ is any neighborhood of $x_i$, then there are Cantor sets $Z_i\subset A_i$ 
such that the following conditions hold; \\
$(1)$~every tuple of finite points in the Cantor set $Z=\cup_i Z_i$ is an $IE$-tuple; \\
$(2)$~for all $k \in \N$, $k$  distinct points $y_1,y_2,...,y_k\in Z$ and any points $z_1,z_2,...,z_k\in Z$, one has 
$$\liminf_{n\to \infty} \max\{d(f^n(y_i),z_i)|~1\leq i\leq k\}=0. $$  In particular, $Z$ is a $\delta$-scrambled set of $f$ for some $\delta >0$. 
\end{thm}
In [6],  by use of local entropy theory (IE-tuples), we proved the following theorem. 

\begin{thm}{\rm ([6])}
Suppose that $G$ is  any graph and $f:X \to X$ is a  homeomorphism on a $G$-like continuum $X$ with   positive topological entropy. Then $X$ contains an indecomposable subcontinuum. Moreover, if $G$ is a tree,  there is a pair of two distinct points $x$ and $y$ of $X$ such that the pair $(x,y)$ is an IE-pair of $f$ and 
 the irreducible continuum between $x$ and $y$ in $X$ is an indecomposable  subcontinuum. 
\end{thm}

The next theorem is a structure theorem for positive topological  entropy homeomorphisms on $G$-like continua. 
The result  is the main theorem in this paper which implies that for any graph $G$, a positive topological  entropy homeomorphism on a $G$-like continuum $X$ admits Cantor set $Z$ which  yields both some complicated structures in topology and the chaotic behaviors of  Kerr and Li [18] in dynamical systems. Especially, the Cantor set  $Z$ has the freely tracing property by free chains.

\begin{thm}  Let $G$ be any graph, $X$ a $G$-like continuum and $f:X \to X$ a  homeomorphism on $X$ with   positive topological entropy. Suppose that  $x_1,x_2,...,x_m~(m\geq 2)$ are finite distinct points of $X$ such that the tuple $(x_1,x_2,...,x_m)$ is an $IE$-tuple of $f$ and $A_i~(i=1,2,...,m)$ is any neighborhood of $x_i$. Then  
there are Cantor sets $Z_i \subset A_i$ and an indecomposable subcontinuum $H$ of $X$ 
such that the following conditions hold;\\
$(1)$ the Cantor set  $Z=\cup_{i=1}^mZ_i$ 
is vertically embedded to composants of $H$; i.e., if $x, y$ are distinct points of $Z$, then the irreducible  continuum $Ir(x,y;H)$ between $x$ and $y$ in $H$ is $H$,
\\
$(2)$~$Z$ has the freely tracing property by free chains; for any $m\in \N$ and any order $x_1\to x_2 \to \cdots \to x_m$ of $m$ distinct points 
$x_{i}~(i=1,2,...,m)$ of $Z$ and any $\epsilon >0$, there is an open cover ${\mathcal U}$ of $X$ such that the  
mesh of ${\mathcal U}$ is less than $\epsilon$, the nerve  $N({\mathcal U})$ of ${\mathcal U}$ is $G$ and there is a free chain in ${\mathcal U}$ which follows from the pattern $[x_1\to x_2 \to \cdots \to x_m],$ \\
$(3)$~every tuple of finite points in the Cantor set $Z$ is an $IE$-tuple of $f$, and \\
$(4)$~for all $k \in \N$, any  distinct $k$ points $y_1,y_2,...,y_k\in Z$ and any points $z_1,z_2,...,z_k\in Z$, the following condition  holds
$$\liminf_{n\to \infty} \max\{d(f^n(y_i),z_i)|~1\leq i\leq k\}=0. $$  In particular, $Z$ is a $\delta$-scrambled set of $f$ for some $\delta >0$. 
\end{thm}

In the statement of Theorem 3.3, we need the condition that $X$ is a $G$-like continuum for a graph $G$. 
\\

Example 2. Let $g: Z\to Z$ be a homeomorphism on a Cantor set $Z$ which has  positive topological entropy. Let $X=$~Cone$(Z)$ be the cone  of $Z$ and let $f:X\to X$ be a homeomorphism which is the natural extension of $g$. Then $h(f)>0$ and $X$ is tree-like, but  $X$ is not $G$-like for any graph $G$. Note that $X$ contains no indecomposable subcontinuum. Also, if $D$ is a subset with $|D|\geq 3$, then $D$ does not have the freely tracing property by chains.
\\

 We will freely use the following facts from the local entropy theory. 

\begin{prop}{\rm ([18, Propositions, 3.8, 3.9])}
Let $X$ be a compact metric space and let $f: X \rightarrow X$ be a map.
\begin{enumerate}
\item
Let $(A_1, \ldots, A_k)$ be a tuple  of closed subsets of $X$ which has an independent set
of positive density. Then, there is an IE-tuple $(x_1, \ldots, x_k)$ with $x_i \in A_i$
for $1 \le i \le k$. 
\item
 $h(f) >0$ if and only if $f$ has an IE-pair $(x_1, x_2)$ with $x_1 \neq x_2$.
\item $IE_k$ is closed and $f \times \ldots \times f$ invariant subset of $X^k$.
\item 
If $(A_1, \ldots, A_k)$ has an  independence set
with positive density and, for $1 \le i \le k$,
${\mathcal A}_i$ is a finite collection of sets such that $A_i \subseteq \cup{\mathcal A}_i$, 
then there is $A'_i \in {\mathcal A}_i$ such that $(A'_1, \ldots, A'_k)$ has an  independence set
with positive density. 
\end{enumerate}
\end{prop}

To prove Theorem 3.3, we need the following results.

\begin{prop}{\rm ([6, Proposition 3.1])}
Let $I \subseteq \N$ be a set with positive density and $n \in \N$. Then,
there is a finite set $F \subseteq I$ with $|F|=n$ and a positive density set $B$ such that
$F+ B \subseteq I$.
\end{prop}

\begin{prop}{\rm ([6, Proposition 3.2])}
Let $X$ be a compact metric space and let  $f:X \rightarrow X$ be a map.
Let ${\mathcal A}$ be a collection which has an  independence set with positive density and $n \in \N$. Then, there is a 
finite set $F$ with $|F|=n$ such that
\[{\mathcal A}_F = \{ \bigcap_{i \in F} f^{-i}(Y_i): Y_i \in {\mathcal A}\}
\]
has an is an independence set with positive density.
\end{prop}

Let $m\geq 2$ and let  $\{1,2,...,m\}^n$ be the set of all functions from $\{1,2,...,n\}$ to $\{1,2,...,m\}$.
For $\sigma \in \{1,2,...,m\}^n~(m\geq 2)$, we write $\sigma=(\sigma(1),\sigma(2),...\sigma(n))$,  where $\sigma(i) \in \{1,2,...,m\}$. Note that $|\{1,2,...,m\}^n|=m^n$.

\begin{prop}{\rm (cf. [6, Proposition 3.3])}
 Let $ m, n \in \N$,  and $\sigma_1, \ldots, \sigma_{[(m-1)n+1][(m-1)^n+1]} $ be any sequence of distinct elements of $\{1,...,m\}^n$.
Then there are $1\leq i \leq n$ and $$ \ 1\le k_1<k_2 < k_3 <  \ldots < k_{m} \leq [(m-1)n+1][(m-1)^n+1]$$ such that 
$\sigma_{k_j}(i) =j$ for $j=1,...,m$.
\end{prop}

\begin{proof}
First, we prove the following claim $(*)$: If $B$ is a subset of $\{1,2,...,m\}^n$ with 
$|B|=(m-1)^n+1$, then  
there is  $1\leq i \leq n$ such that 
$$B(i)=\{1,2,...,m\}, $$  where 
$B(i)=\{\sigma(i)|~\sigma \in B\}$.

Suppose, on the contrary that 
for each $1\leq j\leq n$, 
$$|B(j)|\leq m-1,$$ where 
$B(j)=\{\sigma(j)|~\sigma \in B\}.$ 
Then we may consider that the set $B$ is  a subset of $$B(1)\times B(2) \times \cdots \times B(n)$$ whose cardinality is $\leq (m-1)^n$. This is a contradiction. 

Next, we will prove this proposition. We divide the given  sequence 
$$\sigma_1, \ldots, \sigma_{[(m-1)n+1][(m-1)^n+1]}$$ into $(m-1)n+1$ subsequences  as follows. Let 
$$B_1=\{\sigma_1, \ldots, \sigma_{(m-1)^n+1}\},$$
$$B_2=\{\sigma_{(m-1)^n+2}, \ldots, \sigma_{2[(m-1)^n+1]}\},$$
$$\cdots $$
$$B_{(m-1)n+1}=\{\sigma_{[(m-1)n][(m-1)^n+1]+1}, \ldots, \sigma_{[(m-1)n+1][(m-1)^n+1]}\}.$$
Since $|B_s|= (m-1)^n+1
~(s=1,2,..,(m-1)n+1)$, the above claim $(*)$ implies that for each $B_s$, 
there is $1\leq i_s \leq n$ such that $B_s(i_s)=\{\sigma (i_s)|~\sigma \in B_s\}=\{1,..,m\}$. Define a function
$F:{\mathcal B}=\{B_s|~1 \leq s \leq (m-1)n+1\} \to \{1,2,..,n\}$ by $F(B_s)=i_s$. Since 
$|{\mathcal B}|=(m-1)n+1$, we can find $1\leq i \leq n $ such that $|F^{-1}(i)|\geq m$. Then we can choose 
$1\leq s_1< s_2 < \cdots < s_m \leq (m-1)n+1$ such that 
$B_{s_j}(i)=\{1,2,..,m\}~(j=1,2,...,m)$. By use of this fact, we can choose  $\sigma_{k_j}\in B_{s_j}$ such that $\sigma_{k_j}(i) =j$ for $j=1,...,m$. Then 
$$ \ 1\leq k_1 < k_2 <  \ldots < k_{m} \leq [(m-1)n+1][(m-1)^n+1]$$ and  $\sigma_{k_j}(i) =j$ for $j=1,...,m$.
\end{proof}

To check the chaotic behaviors of Kerr and Li ([18, Theorem 3.18]), we need the following lemma.

\begin{lem}
Let  $f: X \rightarrow X$ be a map of a compact metric space $X$.
Suppose that $(A_1, \ldots, A_k)$ is a tuple  of closed subsets of $X$ which has an independent set
of positive density. Then there is a tuple $(A'_1, \ldots, A'_k)$ of closed subsets of $X$ which has an independent set with positive density such that $A'_j \subset A_j~(j=1,2,...,k)$,  and if $h:\{1,2,...,k\}\to \{1,2,...,k\}$ is any function, then there is $n_h\in \N$ such that $f^{n_h}(A'_j)\subset A_{h(j)}$ for each $j=1,2,...,k$.  
\end{lem}

\begin{proof} Suppose that  $(A_1, \ldots, A_k)$ has an independent set $I \subset \N$ 
with positive density satisfying the above condition $(kl)$.  Let $K$ be the set $\{1,2,\cdots,k\}^{k}$ of all functions on $\{1,2,...,k\}$.
Since $|K|=k^k(=p)$, we can put $K=\{h_1,h_2,\cdots,h_p\}$. 
By Proposition 3.5, there is a finite set $F\subset I$ with $|F|=p$ and  a positive density set $B$ such that $F+B\subset I$. 
Let $F=\{i_1,i_2,\cdots,i_p\}$. For each $1\leq j\leq k$, we 
put  \[A'_j=A_j\cap \bigcap_{i_s \in F} f^{-i_s}(A_{h_s(j)})\]
Then ${\mathcal A'}=\{A'_j|~1\leq j\leq k\}$ is a desired family. In fact, $f^{i_s}(A'_j)\subset A_{h_s(j)}$  for each $j=1,2,...,k$.
\end{proof}

\begin{prop}
Let  $X$ be a $G$-like continuum for a graph $G$ and let $T$ be a Cantor set in $X$. Suppose that 
$T$ has the freely tracing property by chains. 
Then 
any minimal continuum $H$ in $X$ containing $T$ is indecomposable and there is $s\in \N$ such that 
for any composant $c$ of $H$, $|c\cap T|\leq s$. In particular, no infinite points of $T$ belong to the same composant of $H$. Also, there is a subset $Z$ of $T$ such that $Z$ is a Cantor set and $Z$ is vertically embedded to composants of $H$. 
\end{prop}

\begin{proof} For the graph $G$, we can find sufficiently large $s\in \N$ such that $G$ does not contain $s$ simple closed curves. 

 Consider the family $\mathcal{K}$ of all subcontinua of $X$ containing $T$. By the  the Zorn's lemma, there is a minimal element $H$ of  $\mathcal{K}$.
We will show that $H$ is indecomposable. 
Suppose, on the contrary, that $H$ is decomposable. There are two proper subcontinua 
$A$ and $B$ of $H$ such that $H=A\cup B$.
Since $H$ is a minimal continuum (=irreducible continuum) containing $T$,  there are two points 
$x,y \in T$ with $x\in A-B, y\in B-A$. Since $T$ 
is perfect, $T\cap (A-B)$ and $T\cap (B-A)$ are  infinite sets. Then there are distinct points $a_i \in  T\cap (A-B)~(i=0,1,2,...,s), b_i \in  
T\cap (B-A)~(i=1,2,...,s)$. Let $\epsilon >0$ be sufficiently small so that $d(A,\{b_i|~i=1,2,...,s\})>\epsilon$. 
Since $T$ has the the freely tracing property by chains, there is an open cover ${\mathcal U}$ of $X$ such that the  
mesh of ${\mathcal U}$ is less than $\epsilon$, the nerve  $N({\mathcal U})$ of ${\mathcal U}$ is $G$ and there is a chain in ${\mathcal U}$ which follows from the pattern $$[a_0\to b_1 \to a_1 \to b_1 \to a_2 \to b_2 \to \cdots \to a_{s-1}\to b_s \to a_s].$$
Since $A$ is connected, for any two points $z, z'\in A$ there a chain $\{C_1,C\2,...,C_n\}$ in  ${\mathcal U}$ from 
$z$ to $z'$ such that $C_j\cap A\neq \emptyset$. 
By use of these facts, we can easily find 
distinct $s$ simple closed curves in $N({\mathcal U})=G$. This is a contradiction. By the similar arguments, we see that for any composant $c$ of $H$, $|c\cap T|\leq s$. To find a desired Cantor set $Z\subset T$, we consider the following subset of $T^2$:
$$R=\{(x,y)\in T^2|~\mbox{there is a proper subcontinuum}~F \mbox{ of}~H~\mbox{with}~ x,y \in F\}$$

Let $\{U_i|~i\in \N\}$ be an open base of $H$ and let 
$$R_i=\{(x,y) \in T^2|~\mbox{there is a subcontinuum}~F~\mbox{ in }~H-U_i~\mbox{with}~ x,y \in F\}.$$ Note that $R_i$ is a closed set of $T^2$ and 
$$R=\bigcup \{R_i|~i\in \N\}.$$ 
Since each composant of $H$ contains no infinite points of $T$, we see that  $R$ is a nowhere dense $F_{\sigma}$-set in $T^2$ (see 
 [21, p. 71, Application 2]). By 
 [21, p. 70, Corollary 3], there is a Cantor set $Z$ in $T$ such that $Z$ is independent in $R$, i.e., if $x,y \in Z$ and $x \neq y$, then $(x,y)\notin R$. Then we see that 
$Z$ is vertically embedded to composants of $H$. 
This completes the proof. 
\end{proof}

The following lemma is the key lemma to prove the  main theorem. 

\begin{lem}
Let $G$ be a graph and let $f$ be a homeomorphism on a $G$-like continuum $X$ with positive topological entropy. 
 Suppose that ${\mathcal A}$ is a finite open collection of $X$ which has an independence set of $f$ with positive density, any distinct elements of  ${\mathcal A}$ are disjoint,  and $|{\mathcal A}|=m\geq 2$. Then for any $\epsilon >0$ and any  order $A_1\to A_2\to \cdots \to A_{m}$ of all elements of ${\mathcal A}$,  
there exists a finite open cover ${\mathcal V}$ of $X$ satisfying the following conditions; \\
$(1)$~the mesh of ~${\mathcal V}$ is less than $\epsilon$,\\
$(2)$~the nerve $N({\mathcal V})$ of~ ${\mathcal V}$ is $G$, \\
$(3)$~for each $A\in  {\mathcal A}$ there is a shrink  $s(A)\in {\mathcal V}$  with $s(A)\subset A$ such that $$s({\mathcal A})=\{s(A)|~A\in  {\mathcal A}\}$$ has an independence set with positive density, and \\
$(4)$~there is a free chain $[s(A_1)\to s(A_2) \to \cdots  \to s(A_{m})]$ 
from $s(A_1)$ to $s(A_{m})$ in ${\mathcal V}$ which  follows from the pattern $[A_1\to A_2\to\cdots \to A_{m}]$. 
\end{lem}

\begin{proof} 
 We put $${\mathcal A}=\{A_1,A_2,\cdots, A_{m}\}.$$
For each $A\in {\mathcal A}$, we can choose an open set $A' \subset \overline{A'} \subset A$ so that $$ {\mathcal A'}=\{A'|A\in  {\mathcal A}\}$$ has an independence set $I$  with positive density (see Proposition 3.4).
We choose a sufficiently small positive number $\epsilon'< \epsilon$ so that $d(\overline{A'},X-A)>\epsilon'$ for each $A\in  {\mathcal A}$. 
Let $E(G)$ be the set of all edges of the graph $G$ and let $|E(G)|$ be the cardinality of the set $E(G)$.   
 We can choose a sufficiently large natural number $n\in \N$ such that 
$$m^n>|E(G)|\cdot [(m-1)n+1][(m-1)^n+1].$$
By Proposition 3.5, we have $F$ with $|F|=n$ and $B$ satisfying the condition of Proposition 3.5. 
Put $F=\{j_1,j_2,...,j_n\}$.
Recall 
\[{\mathcal A'}_F = \{ \bigcap_{j \in F} f^{-j}(Y_j)~|~ Y_j \in {\mathcal A'}\}=\{ \bigcap_{i=1}^n f^{-j_i}(A_{\sigma (i)})~|~\sigma \in \{1,2,...,m\}^n\}
.\]
Note that $|{\mathcal A'}_F|=m^n$ and 
${\mathcal A'}_F$  has the 
 independence set $B$ with positive density. 
Since $F$ is a finite set,  we can choose a sufficiently small 
positive number $\delta >0$ such that every pair of distinct elements of  ${\mathcal A'}_F$ is at least $\delta$ apart and 
if $U$ is a subset of $X$ whose diameter is less than $\delta$, then the diameter  of $f^i(U)~(i\in F)$ is less than $\epsilon'$. Since $X$ is $G$-like, we can choose an open cover ${\mathcal U}$ of $X$ such that $N({\mathcal U})$ is 
$G$ and the mesh of ${\mathcal U}$ is less than $\delta$. Since $\delta$ is so small, we see that each element of ${\mathcal U}$ intersects at most one element of 
${\mathcal A'}_F$. By Proposition 3.4 (4),  we obtain a subcollection  ${\mathcal U'}$ of ${\mathcal U}$ such that 
each element of ${\mathcal A'}_F$ intersects with only one element of  ${\mathcal U'}$, 
 $|{\mathcal U'}|=m^n$ and  
 the family ${\mathcal U'}$ 
has an independent set of positive density.  Then we can choose a free chain 
${\mathcal C}$ in ${\mathcal U}$ such that ${\mathcal C}$ contains at least 
$$m^n/|E(G)|\geq [(m-1)n+1][(m-1)^n+1]$$
 many elements of ${\mathcal U'}$. 
Put ${\mathcal C}=\{C_1,C_2,\cdots,C_p\}$. 
 Note that each $U'\in {\mathcal U'}$ determines the element 
$\sigma \in \{1,2,...,m\}^n$. By Proposition 3.7, we can choose $i\in F$ such that there is a sequence $$ \ 1\le k_1<k_2 < k_3 <  \ldots < k_{m} \leq [(m-1)n+1][(m-1)^n+1]$$ such that $C_{k_j}\in {\mathcal U'}$ and 
$$f^i(C_{k_j})\cap A'_j\neq \emptyset$$ 
for each $j=1,2,...,m.$
By the choice of $\epsilon'$, $f^i(C_{k_j})\subset A_j$
for each $j=1,2,...,m.$
Then the free chain $$[f^i(C_{k_1})\to  f^i(C_{k_2})\to \cdots \to f^i(C_{k_m})]$$ in $f^i({\mathcal U})$ 
follows from the pattern 
$[A_1\to A_2\to\cdots \to A_{m}].$ 
Put $s(A_j)=f^i(C_{k_j})$ and ${\mathcal V}=f^i({\mathcal U})$. 
Then $$s({\mathcal A})=\{s(A)|~A\in  {\mathcal A}\}$$
 is the desired family.
This completes the proof. 
\end{proof}

Now, we will prove  Theorem 3.3. 

\begin{proof} 
Let ${\mathcal A_1}=\{A_i|~i=1,2,...,m\}$. Then we may assume that  ${\mathcal A_1}$ has an independence set with positive density and the closures of any distinct elements of  ${\mathcal A_1}$ are disjoint. Also, we may assume $|{\mathcal A_1}|=m~(=m_1)\geq 3$ and 
the mesh of ${\mathcal A_1}$ is less than $\epsilon_1\in 
(0, 1/2)$. 
By Lemma 3.8, we have a collection 
$${\mathcal A'_1}=\{A_1',A_2',\cdots,A_{m_1}'\}$$ 
of open sets which has an independent set with positive density and satisfies the following condition (KL) for ${\mathcal A_1}$;\\

~~ (KL)~$A_i' \subset A_i~(i=1,2,...,m_1)$,  and if $h:\{1,2,...,m_1\}\to \{1,2,...,m_1\}$ is any function, then there is
$n_h\in \N$ such that $f^{n_h}(A_i')\subset A_{h(i)}$ for each $i=1,2,...,m_1$.  \\

 Consider the set ${\mathcal A'_1}(m_1)$ of 
 orders (=permutations) of   all elements 
of ${\mathcal A'_1}$. 
 Note that the cardinality of ${\mathcal A'_1}(m_1)$ is 
$m_1!=m_1\cdot (m_1-1)\cdot \cdots 2\cdot 1$. 
We consider the set $Ord~ {\mathcal A'_1}(m_1)$ of equivalence classes of  elements
 of ${\mathcal A'_1}(m_1)$, i.e.,    
$$Ord~ {\mathcal A'_1}(m_1)=\{[A^i_1\to A^i_2\to \cdots  \to A^i_{m_1}]~|~ i=1,2,...,q_1\},$$
where $q_1=m_1!/2$.
By Lemma 3.10,  
there exists a finite open cover ${\mathcal U}_1$ of $X$ such that the mesh of  ${\mathcal U}_1$ is less than $\epsilon_1$ 
and ${\mathcal U}_1$ satisfies the following conditions; \\ $(1)$~the nerve $N({\mathcal U}_1)$ of~ ${\mathcal U}_1$ is homeomorphic to $G$, \\
$(2)$~for each $A\in  {\mathcal A_1}$ there is $s_1(A)\in {\mathcal U}_1$ such that  $s_1(A)\subset A' \subset A$, the family  $${\mathcal A_1(1)}=\{s_1(A)|~A\in  {\mathcal A_1}\}$$ has an independence set with positive density, and we have a free chain $$[s_1(A^1_1)\to s_1(A^1_2)\to  \cdots \to s_1(A^1_{m_1})]$$ to 
from $s_1(A^1_1)$ to $s_1(A^1_{m_1})$ in ${\mathcal U}_1$ which follows from the pattern $$[A^1_1\to A^1_2\to \cdots  \to A^1_{m_1}].$$
This is the case $i=1$.
If we continue this procedure by induction on $i=1,2,...,q_1$, we obtain a sequence  
  ${\mathcal U}_1, {\mathcal U}_2,\cdots, {\mathcal U}_{q_1}$ of finite open 
covers of $X$ and $s_i(A)\in {\mathcal U}_i~( A \in 
{\mathcal A_1}, i=1,2,...,q_1)$ such that the following conditions hold;
\\
$(3)$~the nerve $N({\mathcal U}_i)$ of~ ${\mathcal U}_i$ is homeomorphic to $G$, \\
$(4)$~${\mathcal U}_{i+1}$ is a refinement of ${\mathcal U}_{i}$, \\
$(5)$~for each $A\in  {\mathcal A_1}$, $s_i(A) \in  {\mathcal U}_i~
(i=1,2,...,q_1)$ satisfies that 
$A \supset A' \supset s_{i}(A)\supset s_{i+1}(A)$ and the family  $${\mathcal A_1(i)}=\{s_i(A)|~A\in  {\mathcal A_1}\}~(i=1,2,...,q_1)$$ is an independence set with positive density, and there is a free chain 
$$[s_{i}(A^i_1)\to s_{i}(A^i_2)\to \cdots 
\to s_{i}(A^i_{m_1})]$$
from $s_{i}(A^i_1)$ to $s_{i}(A^i_{m_1})$ in ${\mathcal U}_{i}$ follows from the pattern 
$$[s_{i-1}(A^i_1)\to s_{i-1}(A^i_2)\to \cdots \to s_{i-1}(A^i_{m_1})]
~(i=1,2,...,q_1),$$
 where $s_0(A^1_j)=A^1_j$, 
 etc.  

By Proposition 3.6, for each $A\in  {\mathcal A_1}(q_1)$,
 we can choose nonempty  open sets $s_{q_1}(A)^+$ and $s_{q_1}(A)^{-}$ in $s_{q_1}(A)$ such that $\overline{s_{q_1}(A)^+}\cap \overline{s_{q_1}(A)^-}=\emptyset$ and  the collection 
$${\mathcal A_2}=\{s_{q_1}(A)^+,s_{q_1}(A)^-|~A\in {\mathcal A_1(q_1)}\}$$
has an independence set with positive density. 

Let 
$|{\mathcal A_2}|=m_2~(=2m_1)$ and $0<\epsilon_2 
\leq  \frac{1}{2}\cdot \epsilon_1$. 
By Lemma 3.8, for ${\mathcal A_2}$ we can choose a collection ${\mathcal A'_2}$ such that the mesh  of ${\mathcal A'_2}$ is less than $\epsilon_2$  and ${\mathcal A'_2}$ satisfies the condition (KL) 
 for  ${\mathcal A_2}$ as above.  Also, we consider the set ${\mathcal A'_2}(m_2)$ of 
 permutations of   all elements 
of ${\mathcal A'_2}$ and the set $Ord~ {\mathcal A'_2}(m_2)$ as above.

By repeated use of Lemma 3.10, we obtain  desired families  $${\mathcal A_2(i)}=\{s_i(A)|~A\in  {\mathcal A_2}\}~(i=1,2,...,q_2)$$
as above, where $q_2=m_2!/2$. By use of ${\mathcal A_2(q_2)}$, we obtain ${\mathcal A_3}$ 
as above. Note that $|{\mathcal A_3}|=m_3~(=2^2\cdot m_1)$.  

If we continue this procedure, we have a sequence 
$\epsilon_i~(i\in \N)$ of positive numbers 
 and  sequences of of  families ${\mathcal A_i}$ and ${\mathcal A'_i}$ of open sets of $X$ satisfying the following conditions;\\
$(6)$~$\epsilon_i>\epsilon_{i+1}~(i\in \N)$ and $\lim_{i\to \infty} \epsilon_i=0$,\\
$(7)$~the closures of any distinct elements of  ${\mathcal A_i}$ are disjoint,  
each $A\in {\mathcal A_i}$ contains the closures of two elements of ${\mathcal A_{i+1}}$ and 
 the mesh of  ${\mathcal A_i}$ is less than $\epsilon_i$, \\
 $(8)$~${\mathcal A_i}$ and ${\mathcal A'_i}$ have   independence sets with positive density,   \\
$(9)$~${\mathcal A'_i}$ satisfies the condition  (KL) for ${\mathcal A_i}$, and \\
$(10)$~for any order $E_1\to E_2\to \cdots \to E_{m_i}$  
of all elements of ${\mathcal A_i}$, there is an open cover 
${\mathcal U}$ of $X$ such that 
the mesh of ${\mathcal U}$ is less than $\epsilon_i$,  $N({\mathcal U})$ is $G$ and there is a free chain in ${\mathcal U}$ which follows from $$[E_1\to E_2\to \cdots \to E_{m_i}].$$

For each $i \in \N$, we put $$T_i=\bigcup {\mathcal A_i}~(i\in \N)~  \mbox{ and}$$ 
$$T=\bigcap_{i\in \N} T_i.$$
 Then $T$ is a Cantor set. By the above constructions, we see that for any $k\in \N$ and any order $$x_1\to x_2 \to \cdots \to x_k$$ of $k$ distinct points 
$x_{i}~(i=1,2,...,k)$ of $T$ and any $\epsilon >0$, there is an open cover ${\mathcal U}$ of $X$ such that the  
mesh of ${\mathcal U}$ is less than $\epsilon$, the nerve  $N({\mathcal U})$ of ${\mathcal U}$ is $G$ and there is a free chain in ${\mathcal U}$ which follows from the pattern $$[x_1\to x_2 \to \cdots \to x_k].$$ By Proposition 3.9, 
any minimal continuum $H$ in $X$ containing $T$ is indecomposable and 
 no infinite points of $T$ belong to the same composant of $H$. Also, by Proposition 3.9, we can choose a subset $Z$ of $T$ such that $Z$ is a Cantor set and $Z$ is vertically embedded to composants of $H$. 
Also, by the constructions, we see that  $T$ satisfies the 
 conditions (2), (3) and (4) of Theorem 3.3. Note that any subset of $T$ satisfies the conditions. Hence the Cantor set $Z$ satisfies the conditions of Theorem 3.3. This completes the proof.
\end{proof}

\begin {cor} Let $G$ be any graph. 
If $f:G \to G$ is a positive entropy map on $G$, 
then there exist an indecomposable subcontinuum $H$ of $X= \varprojlim (G,f)$ and a Cantor set $Z$ in $H$ satisfies the following conditions;\\
$(1)$~$Z$ is vertically embedded to composants of $H$,   
 \\
$(2)$~$Z$ has the freely tracing property by free chains, \\
$(3)$~every tuple of finite points in the Cantor set $Z$ is an $IE$-tuple of the shift map 
$\sigma_f$ 
 and \\
$(4)$~for all $k \in \N$, any  distinct $k$ points $y_1,y_2,...,y_k\in Z$ and any points $z_1,z_2,...,z_k\in Z$, the following condition  holds
$$\liminf_{n\to \infty} \max\{d(\sigma_f^n(y_i),z_i)|~1\leq i\leq k\}=0. $$  In particular, $Z$ is a $\delta$-scrambled set of $\sigma_f$ for some $\delta >0$. 
\end{cor}

\begin{proof}
Note that $h(f)=h(\sigma_f)>0$ and $\sigma_f$ is a homeomorphism on the $G$-like continuum $X= \varprojlim (G,f)$. This result follows from Theorem 3.3.  
\end{proof}

For a special case, we have the following.

\begin{cor} Let $X$ be one of the Knaster continuum, 
 solenoids or  
Plykin attractors. If $f$ is any positive topological entropy homeomorphism on $X$, then 
there is a Cantor set $Z$ in $X$ such that the Cantor set $Z$ satisfies the following conditions;\\
$(1)$~$Z$ is vertically embedded to composants of $X$, 
 \\
$(2)$~$Z$ has the freely tracing property by free chains, \\
$(3)$~every tuple of finite points in the Cantor set $Z$ is an $IE$-tuple of $f$, and \\
$(4)$~for all $k \in \N$, any  distinct $k$ points $y_1,y_2,...,y_k\in Z$ and any points $z_1,z_2,...,z_k\in Z$, the following condition  holds
$$\liminf_{n\to \infty} \max\{d(f^n(y_i),z_i)|~1\leq i\leq k\}=0. $$  In particular, $Z$ is a $\delta$-scrambled set of $f$ for some $\delta >0$. 
\end{cor}
\begin{proof} By Theorem 3.3, there is an indecomposable subcontinuum $H$ in $X$. Note that any proper subcontinuum of $X$ is not indecomposable and hence $H=X$. This corollary follows from Theorem 3.3.
\end{proof}

An onto map $f:X \to Y$ of continua is {\it monotone} if for any $y\in Y$, $f^{-1}(y)$ is connected.  In [16], we proved that if $G$ is a graph and $f:X\to X$ is a monotone map on a  $G$-like continuum $X$ which has positive topological entropy, then $X$ contains an indecomposable  subcontinuum. Here we give the following more precise result. 

\begin{thm} Suppose that $G$ is a graph and  $X$ is a $G$-like continuum. 
If $f:X \to X$ is a monotone map  on $X$ with positive topological entropy, then there exist an indecomposable subcontinuum $H$ of $X$ and a Cantor set $Z$ in $H$ such that the Cantor set $Z$ satisfies the following conditions;\\
$(1)$~$Z$ is vertically embedded to composants of $H$, 
 \\
$(2)$~every tuple of finite points in the Cantor set $Z$ is an $IE$-tuple of $f$,   \\
$(3)$~for all $k \in \N$, any  distinct $k$ points $y_1,y_2,...,y_k\in Z$ and any points $z_1,z_2,...,z_k\in Z$, the following condition  holds
$$\liminf_{n\to \infty} \max\{d(f^n(y_i),z_i)|~1\leq i\leq k\}=0. $$  In particular, $Z$ is a $\delta$-scrambled set of $f$ for some $\delta >0$. 
\end{thm}

A continuum $E$ is an $n$-{\it od} $(2\leq n < \infty)$ if $E$ contains a subcontinuum $A$ 
such that 
the complement of $A$ in $E$ is the union $n$ nonempty mutually separated sets, i.e., $$E-A=\bigcup \{E_i|~i=1,2,...,n\}$$ for some subsets $E_i$ satisfying the condition:  
$$\overline{E_i}\cap E_j=\emptyset ~ (i\neq j).$$ 
 For any continuum $X$, let $$T(X)=\sup \{n~|~\mbox{there is an $n$-od in}~ X\}.$$ 
Note that if $X$ is a $G$-like continuum  for a graph $G$, 
then $T(X)<\infty$. 

To prove Theorem 3.13, we need the following lemma.

\begin{lem}{\rm (cf. [16, Lemma 2.3])}
Let $X$ and $Y$ be  continua with $T(X)<\infty$. Suppose that  $f:X\to Y$ is an (onto) monotone map, $H'$ is an indecomposable subcontinuum of $X$ and $Z'$ is a Cantor set which is vertically embedded to composants of $H'$.  If  $H=f(H')$ is nondegenerate, then $H$ is an indecomposable 
subcontinuum of $Y$ and there is a subset  $Z$  of $f(Z')$ such that $Z$ is a Cantor set and $Z$ is vertically embedded to composants of $H$.
\end{lem}

\begin{proof}
Note that if  $f: X\to Y$ is monotone, then $f^{-1}(C)$ is connected for any subcontinuum $C$ in $Y$. By use of this fact,  we can see  that $T(Y)\leq T(X)<\infty$. 
For each $x\in Z'$, let $c(x)$ denote the composant of $H'$ containing $x\in Z'$. 
Let $$\mbox{Comp}(Z';H')=\{c(x)|~x\in Z'\}.$$
Since $Z'$ is vertically embedded to composants 
of $H'$, $\mbox{Comp}(Z';H')$ is a family of mutually disjoint  dense connected subsets $c(x)$ of $H'$. 
 For each $x, y\in Z'$, we define $x  \sim_f y$ provided that  
$f(c(x))\cap f(c(y))\neq \emptyset $.  Also, we define $x  \sim y$ provided that there is a finite sequence 
$x=x_1, x_2,...,x_s=y$ of $x_i\in Z'$ such that $x_i  \sim_f x_{i+1}$ 
for each $i=1,2,..,s-1$. Then the relation $\sim$ is an equivalence relation on $Z'$. 
Note that $f(c(x))\cap f(c(y))\neq \emptyset $ if and only if there is a point $z\in Y$ with 
$$f^{-1}(z)\cap c(x) \neq \emptyset \neq f^{-1}(z)\cap c(y).$$ 
Let $[x]$ denote the equivalence class containing $x \in Z'$, i.e., $$[x]=\{y\in Z'|~ x \sim y\}.$$ Since $f^{-1}(z)$ is a subcontinuum of $X$ for each $z\in Y$, we can conclude that $|[x]|\leq T(X)$. In particular, $f|Z':Z'\to f(Z')$ is a finite-one map and hence $f(Z')$ is a perfect set, i.e., $f(Z')$ has no isolated point. 

Since $Z'$ is an uncountable set, we we can choose an uncountable subset $Z"$ of $Z'$ such that the  family $$\{f(c(x))|~x\in Z"\}$$ is a 
 family of mutually disjoint subsets of $H=f(H')$. 

We will prove that $H=f(H')$ is indecomposable. Suppose, on the contrary, that $H$ is decomposable. There is a proper 
 subcontinuum $A$ of $H=f(H')$ with $$Int_{H}(A)\neq \emptyset.$$ 
Since each composant of $H'$ is dense in $H'$ and hence 
$f(c(x))$ is dense in $H$ for any $x\in Z"$, $f(c(x))\cap A\neq \emptyset$. This implies that  $|T(Y)|=\infty$. This is a contradiction. Hence $H=f(H')$ is indecomposable. 

We show that for each composant $c$ of $H$, $|c\cap f(Z')| \leq T(X)$. Suppose, on the contrary, that there is a proper subcontinuum $C$ of $H$ such that $|C\cap f(Z')|\geq T(X)+1$. Then $f^{-1}(C)$ is a continuum which intersects $T(X)+1$ composants of $H'$. This is a contradiction. 

Since $f(Z')$ is perfect, by the proof of Proposition 3.9, we can find a Cantor set $Z$ in 
$f(Z')$ such that $Z$ is vertically embedded to composants of $H$. 
\end{proof}

We will give the proof of Theorem 3.13. 

\begin{proof}
We consider the inverse $\tilde{f}:\varprojlim (X,f) \to 
\varprojlim (X,f)$ of the shift map $\sigma_f$, i.e., 
$$\tilde{f}(x_1,x_2,x_3,\cdots)=(f(x_1),x_1,x_2,\cdots).$$ 
Note that $h(f)=h(\tilde{f})>0$. By Theorem 3.3, we can find an indecomposable subcontinuum $H'$ and a Cantor set $Z'$ in $\varprojlim (X,f)$ as in Theorem 3.3. 
Since $f$ is a monotone map, we see that the projection $p_n: \varprojlim (X,f) \to X_n=X$ to the $n$-th coordinate $X_n$  is also  monotone. 
 If we choose sufficiently large natural number $n$, then $H=p_n(H')$ is nondegenerate. By the above lemma, $H$ is indecomposable and there is a Cantor set $Z\subset p_n(Z')$ such that $Z$ is vertically embedded to composants of $H$. Note that the projection  $p_n$ 
preserves the properties of $IE$-tuples  and (4) of Theorem 3.3.  Then  we see that $H$ and $Z$ are desired spaces. 
\end{proof}

 \section{Chaotic continua of continuum-wise expansive homeomorphisms and IE-tuples} 
\quad\ 
In this section, we study dynamical behaviors of continuum-wise expansive homeomorphisms related to IE-tuples and chaotic continua in topology. Any continuum-wise expansive 
homeomorphism $f$ on a continuum $X$ has positive topological entropy and hence $f$ has IE-tuples (see Theorem 4.1 below). Also,    $X$ contains a chaotic continuum and chaotic continuum has  uncountable mutually disjoint (unstable or) stable dense connected $F_{\sigma}$-sets (see Theorem 4.1).  
In this section, we study some precise results of IE-tuples  related to 
(unstable) stable connected sets of chaotic continua  and composants of indecomposable  continua.

A homeomorphism $f:X \to X$ of a compact metric space $X$ with metric $d$ 
is called {\it expansive} ([5,13]) if there is $c > 0$ such that 
for any $x,y \in X$ and $x \not = y$, then there is an 
integer $n \in \Z$ such that 
\begin{center}
$d(f^{n}(x),f^{n}(y)) > c$.
\end{center}
A homeomorphism $f:X \to X$ of a compact metric space $X$ is 
{\it continuum-wise expansive} (resp. {\it positively continuum-wise expansive})  [15] if there is $c > 0$ such that if $A$ 
is a nondegenerate subcontinuum of $X$, then there is an integer 
$n \in \Z$ (resp. a positive integer $n \in \N$) such that 
\begin{center}
${\rm diam}\,f^{n}(A) > c$, 
\end{center}
where ${\rm diam}\,B =\sup\{d(x,y)|\,x, y \in B\}$ for a set $B$. 
Such a positive number $c$ is called an {\it expansive constant} for $f$. 
Note that each expansive homeomorphism is continuum-wise expansive, but 
the converse assertion is not true. There are many continuum-wise 
expansive homeomorphisms which are not expansive 
(see [15]). These notions have been extensively studied in the area of 
topological dynamics, ergodic theory  and continuum theory 
(see [5,10-15,27]).  

The hyperspace $2^X$ of $X$ is the set of  all nonempty closed subsets of $X$ 
with the Hausdorff metric $d_{H}$. 
Let 
\begin{center}
$C(X)=\{A \in 2^X|\ A \ \mbox{is connected}\}$.
\end{center}
Note that $2^X$ and $C(X)$ are compact metric spaces (e.g., see [20] and [26]). 
For a homeomorphism $f:X \to X$ and for each closed subset $H$ of $X$ and $x \in H$, 
the {\it continuum-wise 
$\sigma$-stable sets} $V^{\sigma}(x;H)$ ($\sigma=s,u$) of $f$ are defined as follows:  

\begin{center}
$V^{s}(x;H)=\{y \in H| \ \mbox{there is} \ A \in C(H)\ \mbox{such that}
\ x,y \in A \ \mbox{and}\ {\rm lim}_{n\to \infty}{\rm diam}\,f^{n}(A)=0\}$, 
\end{center}
\begin{center}
$V^{u}(x;H)=\{y \in H| \ \mbox{there is} \ A \in C(H)\ \mbox{such that}
\ x,y \in A \ \mbox{and}\ {\rm lim}_{n\to \infty}{\rm diam}\,f^{-n}(A)=0\}$. 
\end{center}
Note that $$V^{s}(x;H) \subset W^s(x)=\{y\in X|~\lim_{n\to \infty}d(f^n(y),f^n(x))=0\}, $$
$$V^{u}(x;H) \subset W^u(x)=\{y\in X|~\lim_{n\to \infty}d(f^{-n}(y),f^{-n}(x))=0\}.$$ 
A subcontinuum $H$ of $X$ is called a 
$\sigma${\it-chaotic continuum} (see [13]) of $f$ (where $\sigma=s,u$) if   
\begin{enumerate} 
\item 
for each $x \in H$, $V^{\sigma}(x;H)$ is dense in $H$, and
\item 
there is $\tau > 0$ such that for each $x \in H$ and each neighborhood 
$U$ of $x$ in $X$, there is $y \in U \cap H$ such that 
\begin{center}
${\rm lim\, inf}_{n \to \infty}d(f^{n}(x),f^{n}(y)) \geq \tau$
in case $\sigma =s$, or 
\end{center}
\begin{center}
${\rm lim \,inf}_{n \to \infty}d(f^{-n}(x),f^{-n}(y)) \geq \tau$ 
in case $\sigma=u$.
\end{center}
\end{enumerate}

We know that if $f: X \to X$ is a continuum-wise expansive homeomorphism,  then 
$V^{\sigma}(z;H)$ is a connected $F_{\sigma}$-set containing $z$. 
If $H$ is a $\sigma$-chaotic continuum of $f$, then the 
decomposition $\{V^{\sigma}(z;H)|\,z \in H\}$ of $H$ is an uncountable 
family of mutually disjoint, dense connected $F_{\sigma}$-sets in $H$. Note that $\sigma$-chaotic continua of $f$ have  very similar structures of composants of indecomposable continua. In fact, for the case of 1-dimensional continua, $\sigma$-chaotic 
continua may be indecomposable (see [10]). 
\vspace{1mm}\\

Example 3. Let $f: T^2\to T^2$ be an Anosov diffeomorphism on the 2-dimensional torus $T^2$, say 
$$\left[
\begin{array}{cc}
2&1\\1&1
\end{array}
\right]
$$
Then $f$ is expansive and $T^2$ itself is a  $\sigma$-chaotic continuum of $f$ for $\sigma=u, s$.  Note that $T^2$ contains no indecomposable $\sigma$-chaotic subcontinuum. 
\vspace{1mm}\\

For continuum-wise expansive homeomorphisms, we have  obtained the following results (see [11,13,15]). 

\begin{thm} Let $f:X \to X$ be a continuum-wise expansive homeomorphism on a continuum $X$. Then the followings hold.
\begin{enumerate}
\item 
 {\rm ([15, Theorem 4.1])} $f$ has positive topological entropy and hence there are IE-tuples. 
\item 
   {\rm ([13, Theorem 3.6 and Theorem 4.1])} There is a $\sigma$-chaotic continuum $H$ of $f$. Moreover, if $H$ is a $u$-chaotic continuum (resp. $s$-chaotic continuum), then there exists a Cantor set $Z$ in $H$ satisfying  the conditions; \\
$(i)$~ no two of points of $Z$ belong to the same $V^{u}(x;H)~(x\in X)$~(resp. $V^{s}(x;H)~(x\in X)$), i.e., $Z$ is vertically embedded to 
$V^{\sigma}(x,H)~(x\in H)$, \\
$(ii)$~$Z$ is a $\delta$-scrambled set of $f^{-1}$ for some $\delta >0$~(resp. $f$). 
\item 
   {\rm ([11, Theorem 2.4])} Moreover, if $f:X \to X$ is a positively continuum-wise expansive homeomorphism, then $X$ contains a $u$-chaotic  continuum $H$ such that $H$ is indecomposable and the set of composants  of $H$ coincides to $\{V^u(x;H)|~x \in H\}$. Also, there exists a Cantor set $Z$ in $H$ satisfying  the conditions;\\ 
$(i)$ $Z$ is vertically embedded to composants 
$V^{u}(x,H)~(x\in H)$, \\
$(ii)$~$Z$ is a $\delta$-scrambled set of $f^{-1}$ for some $\delta >0$. 
\item 
$(d)$   {\rm ([11, Corollary 2.7])} Moreover, if $G$ is any graph and $X$ is
a $G$-like continuum, then $X$ contains a $\sigma$-chaotic  continuum $H$ such that $H$ is indecomposable and the set of composants  of $H$ coincides to $\{V^{\sigma}(x;H)|~x \in H\}$. Moreover if $\sigma=u$~(resp. $s$), then  there exists a Cantor set $Z$ in $H$ satisfying  the conditions;\\ 
$(i)$~ $Z$ is vertically embedded to 
$V^{\sigma}(x,H)~(x\in H)$, \\
$(ii)$~$Z$ is a $\delta$-scrambled set of $f^{-1}$ for some $\delta >0$~(resp. $f$). 
\item 
  {\rm ([11, Theorem 2.6])}  Moreover, if $X$ is a continuum  in the plane $\R^2$,  then 
$X$ contains a $\sigma$-chaotic continuum $H$ of $f$ such that $H$ is indecomposable and the set of composants  of $H$ coincides to $\{V^{\sigma}(x,H)|~x \in H\}$. Moreover if $\sigma=u$~(resp. $s$), then  there exists a Cantor set $Z$ in $H$ satisfying  the conditions;\\ 
$(i)$~ $Z$ is vertically embedded to 
$V^{\sigma}(x,H)~(x\in H)$, \\
$(ii)$~$Z$ is a $\delta$-scrambled set of $f^{-1}$ for some $\delta >0$~(resp. $f$). 
\end{enumerate}
\end{thm}

We consider the case that $\sigma$-chaotic continua are  periodic. By combining Theorem 3.1 and Theorem 4.1, we have the following results. 

\begin{cor} Let  $f:X \to X$ be a continuum-wise expansive homeomorphism on a continuum $X$. Suppose that $X$ contains a periodic  $\sigma$-chaotic
 continuum  $H$  of $f$. Then there exists a Cantor set $Z$ in $H$ such that if $\sigma=u$~(resp. $\sigma=s)$, then the following conditions  hold;\\
$(1)$~$Z$ is vertically embedded to 
$V^{\sigma}(x,H)~(x\in H)$,  \\
$(2)$~every tuple of finite points in the Cantor set $Z$ is an $IE$-tuple of $f^{-1}$~(resp. $f$),  and  
 \\
$(3)$~for all $k \in \N$, any  distinct $k$ points $y_1,y_2,...,y_k\in Z$ and any points $z_1,z_2,...,z_k\in Z$, the following condition  holds
$$\liminf_{n\to \infty} \max\{d(f^{-n}(y_i),z_i)|~1\leq i\leq k\}=0 $$~
$$(resp. \liminf_{n\to \infty} \max\{d(f^{n}(y_i),z_i)|~1\leq i\leq k\}=0).$$ 
\end{cor}

\begin{proof}  We may assume $\sigma=s$. Since the chaotic continuum $H$ is  periodic, there is $i\in \N$ such that $f^{i}(H)=H$. Then $f^{i}|H:H\to H$ is continuum-wise expansive and hence its topological entropy is positive. By Theorem 3.1, there is a Cantor set $Z$ in $H$ as in Theorem 3.1. Since $Z$ is a $\delta$-scrambled set of $f$ for some $\delta>0$, 
$Z$  is vertically embedded to 
$V^{s}(x,H)~(x\in H)$. 
\end{proof}

Similarly, we have the following result. 
 
\begin{cor} Suppose that $f:X \to X$ is a positively continuum-wise expansive homeomorphism on a continuum $X$ such that $X$ has a periodic $u$-chaotic  continuum $H$ which is indecomposable and the set of composants  of $H$ coincides to $\{V^u(x;H)|~x \in H\}$. Then there exists a Cantor set $Z$ in $H$ which is vertically embedded to composants of $H$ and satisfies the conditions;\\
$(1)$~ if $x,y$ belong to the same composant of $H$, then $\lim_{n\to \infty} d(f^{-n}(x),f^{-n}(y))=0$,\\
$(2)$~every tuple of finite points in the Cantor set $Z$ is an $IE$-tuple of $f^{-1}$, and  \\
$(3)$~for all $k \in \N$, any  distinct $k$ points $y_1,y_2,...,y_k\in Z$ and any points $z_1,z_2,...,z_k\in Z$, the following condition  holds
$$\liminf_{n\to \infty} \max\{d(f^{-n}(y_i),z_i)|~1\leq i\leq k\}=0. $$ 
\end{cor}

For special cases, we have the following.

\begin{cor} Suppose that $X$ is one of the  Knaster continuum, Plykin attractors or solenoids. 
If $f:X \to X$ is a continuum-wise expansive homeomorphism on $X$, then $f$ or $f^{-1}$ is  
positively  continuum-wise expansive. In particular, if $f$ is  
positively  continuum-wise expansive, then there exists a Cantor set $Z$ in $X$ such that the Cantor set $Z$  is vertically embedded to composants of $X$ and satisfies the conditions;\\
$(1)$~ if $x,y$ belong to the same composant of $X$, then $\lim_{n\to \infty} d(f^{-n}(x),f^{-n}(y))=0$,\\
$(2)$~every tuple of finite points in the Cantor set $Z$ is an $IE$-tuple of $f^{-1}$,   \\
$(3)$~$Z$ has the freely tracing property by free 
chains, and \\
$(4)$~for all $k \in \N$, any  distinct $k$ points $y_1,y_2,...,y_k\in Z$ and any points $z_1,z_2,...,z_k\in Z$, the following condition  holds
$$\liminf_{n\to \infty} \max\{d(f^{-n}(y_i),z_i)|~1\leq i\leq k\}=0. $$ . 
\end{cor}

\begin{proof} Note that $X$ is a $G$-like continuum for some graph $G$.  
Recall that $X$ is indecomposable and each proper subcontinuum of  $X$ is not indecomposable. Hence a $\sigma$-chaotic continuum of $f$ coincides with $X$.
Then we can easily see that $f$ or $f^{-1}$ is positively continuum-wise expansive. The corollary follows from Theorems 3.3 and 4.1.
\end{proof}

 For the case of the shift map $\sigma_f:\varprojlim (G,f)\to  \varprojlim (G,f)$ of a map $f:G \to G$ on a graph $G$ which has sensitive dependence on initial conditions,  we can find 
a periodic indecomposable $s$-chaotic continuum in $\varprojlim (G,f)$.  Hence  we have the following corollary.

\begin{cor} 
Suppose that $f:G \to G$ is a map on a graph $G$ which has sensitive dependence on initial conditions
 and $\sigma_f:X= \varprojlim (G,f)\to X$ is the shift map of $f$. Then there exists an indecomposable $s$-chaotic continuum $H$ in $X$ such that $\sigma_f^{n}(H)=H$ for some $n\in \N$ and the set of composants of $H$ coincide to $\{V^s(x;H)|x\in H\}$. Hence there is a Cantor set $Z$ in $H$ such that  $Z$ is vertically embedded to composants of $H$ and satisfies the conditions;\\
$(1)$~ if $x,y$ belong to the same composant of $H$, then $\lim_{n\to \infty} d(\sigma_f^n(x),(\sigma_f^n(y))=0$,\\
$(2)$~every tuple of finite points in the Cantor set $Z$ is an $IE$-tuple of $\sigma_f$,   \\
 $(3)$~$Z$ has the freely tracing property by free 
chains, and \\ 
$(4)$~for all $k \in \N$, any  distinct $k$ points $y_1,y_2,...,y_k\in Z$ and any points $z_1,z_2,...,z_k\in Z$, the following condition  holds
$$\liminf_{n\to \infty} \max\{d(\sigma_f^{n}(y_i),z_i)|~1\leq i\leq k\}=0. $$ 
\end{cor}

\begin{proof} Note that $(\sigma_f)^{-1}={\tilde f}$. 
By [11, Corollary 2.8], there is a $s$-chaotic indecomposable continuum $H$ of $\sigma_f$ and a natural number $n$ such that 
$\sigma_f^n(H)=H$ and the set of composants  of $H$ coincides to $\{V^{s}(x,H)|~x \in H\}$. By use of 
Theorem  3.3, we can find a desired Cantor set $Z$ in $H$. 
\end{proof}

Example 4. Let $f: I=[0,1]\to I$ be the map defined by 
$f(t)=4t(1-t)~(t\in I)$.  Note that $f$ has sensitive dependence on initial conditions and $\varprojlim (X,f)$ is the Knaster continuum. 
Then  $\varprojlim (X,f)$ is the  
$s$-chaotic continuum of  the shift homeomorphism $\sigma_f:\varprojlim (X,f)\to \varprojlim (X,f)$ satisfying the conditions  of Corollary 4.5, where $H=\varprojlim (X,f)$. 
\vspace{1mm}\\

For the general case that any $\sigma$-chaotic continua of a continuum-wise expansive homeomorphism are not periodic, we do not know if the statements of Corollaries 4.2 and 4.3 hold. In fact, the following problems remain open.

\begin{ques} Let $f:X \to X$ be a continuum-wise expansive homeomorphism on a continuum $X$. Is it true that there exist 
a $\sigma$-chaotic continuum $H$ of $f$  and a Cantor set $Z$ in $H$ such that $Z$ is vertically embedded to 
$V^{\sigma}(x,H)~(x\in H)$ and satisfies the following conditions? :  \\
If $\sigma=u$~(resp. $\sigma=s)$, then \\
$(1)$~every tuple of finite points in the Cantor set $Z$ is an $IE$-tuple of $f^{-1}$~(resp. $f$), and   \\
$(2)$~for all $k \in \N$, any  distinct $k$ points $y_1,y_2,...,y_k\in Z$ and any points $z_1,z_2,...,z_k\in Z$, the following condition  holds
$$\liminf_{n\to \infty} \max\{d(f^{-n}(y_i),z_i)|~1\leq i\leq k\}=0 $$
$$\mbox{(resp.}~  \liminf_{n\to \infty} \max\{d(f^{n}(y_i),z_i)|~1\leq i\leq k\}=0 \mbox{)}.$$ 
 \end{ques}

\begin{ques} Let $f:X \to X$ be a positively continuum-wise expansive homeomorphism on a continuum $X$. Is it true that there exist 
an indecomposable continuum $H$ and a Cantor set $Z$ in $H$ satisfying the following conditions? : \\
$(1)$~$Z$ is vertically embedded to composants of $H$. \\
$(2)$~If $x,y$ belong to the same composant of $H$, then $\lim_{n\to \infty} d(f^{-n}(x),(f^{-n}(y))=0$.\\
$(3)$~Every tuple of finite points in the Cantor set $Z$ is an $IE$-tuple of $f^{-1}$.  \\
$(4)$~For all $k \in \N$, any  distinct $k$ points $y_1,y_2,...,y_k\in Z$ and any points $z_1,z_2,...,z_k\in Z$, the following condition  holds
$$\liminf_{n\to \infty} \max\{d(f^{-n}(y_i),z_i)|~1\leq i\leq k\}=0. $$  
\end{ques}

{\bf Acknowledgments}. The author would like to thank Dr. Masatoshi Hiraki for useful discussions on Proposition 3.7.

\end{document}